\documentclass[preprint,12pt]{elsarticle}
\usepackage{diagbox}
\usepackage{graphicx}
\usepackage{amssymb}
\usepackage{algcompatible}
\usepackage{algpseudocode}
\usepackage{multirow}
\usepackage{multirow}
\usepackage{xspace}
\usepackage{hyperref}
\usepackage{listings}
\usepackage{url}
\usepackage{csquotes}
\usepackage[mathscr]{euscript} 
\usepackage{amsfonts,epsfig}
\usepackage{enumerate}
\usepackage{amsmath,amscd,amsfonts,amssymb}

\usepackage{lineno,hyperref}
\usepackage{graphicx,subfigure}
\usepackage{amsfonts,epsfig}

\usepackage[mathscr]{euscript} 
\usepackage{enumerate}
\usepackage{amsmath,amscd,amsfonts,amssymb}
\usepackage{times}
\usepackage{extarrows}
\usepackage{stmaryrd}
\usepackage{textcomp}
\usepackage[mathscr]{euscript}
\usepackage{algorithm2e}
\usepackage{colortbl}
\usepackage{color,soul}
\usepackage{newunicodechar}
\usepackage{amsthm}
\theoremstyle{definition}%plain
\newtheorem{thm}{Theorem}%[section]

\newtheorem{defn}{Definition}%
\newtheorem{exa}{Example}
\newtheorem{rem}[thm]{\bf{Remark}}
\newcommand{\cX}{\underline{\bf X}}
\newcommand{\cE}{\underline{\bf E}}
\newcommand{\cU}{\underline{\bf U}}
\newcommand{\cV}{\underline{\bf V}}
\newcommand{\cY}{\underline{\bf Y}}
\newcommand{\A}{{\bf A}}
\newcommand{\X}{{\bf X}}
\newcommand{\Y}{{\bf Y}}
\newcommand{\I}{{\bf I}}
\newcommand{\B}{{\bf B}}
\newcommand{\C}{{\bf C}}

\newcommand{\E}{{\bf E}}
\newcommand{\R}{{\bf R}}
\newcommand{\U}{{\bf U}}

\newcommand{\V}{{\bf V}}

\usepackage{amssymb}

\journal{Journal}

\begin{document}

\begin{frontmatter}

%% Title, authors and addresses

%% use the tnoteref command within \title for footnotes;
%% use the tnotetext command for theassociated footnote;
%% use the fnref command within \author or \address for footnotes;
%% use the fntext command for theassociated footnote;
%% use the corref command within \author for corresponding author footnotes;
%% use the cortext command for theassociated footnote;
%% use the ead command for the email address,
%% and the form \ead[url] for the home page:
%% \title{Title\tnoteref{label1}}
%% \tnotetext[label1]{}
%% \author{Name\corref{cor1}\fnref{label2}}
%% \ead{email address}
%% \ead[url]{home page}
%% \fntext[label2]{}
%% \cortext[cor1]{}
%% \affiliation{organization={},
%%             addressline={},
%%             city={},
%%             postcode={},
%%             state={},
%%             country={}}
%% \fntext[label3]{}

\title{Adaptive Cross Tubal Tensor Approximation}

%% use optional labels to link authors explicitly to addresses:
\author[label1]{Salman Ahmadi-Asl}
\author[label1]{Anh Huy Phan}
\author[label1,label2]{Andrzej Cichocki}
\author[label1]{Anastasia Sozykina}
\author[label3]{Zaher Al Aghbari}
\author[label1]{Jun Wang} 
\author[label1]{and\,\,\,Ivan Oseledets}
\affiliation[label1]{organization={Center for Artificial Intelligence Technology, Skolkovo Institute of Science and Technology, Moscow, Russia,\,\, s.asl@skoltech.ru}
          }
\affiliation[label2]{organization={Systems Research Institute of Polish Academy of Science, Warsaw, Poland}
          }
          \affiliation[label3]{organization={Department of Computer Science, University of Sharjah, Sharjah, 27272, UAE}
          }
%
% \affiliation[label2]{organization={},
%             addressline={},
%             city={},
%             postcode={},
%             state={},
%             country={}}

%\author[inst1]{Author One}

%\affiliation[inst1]{organization={Department One},%Department and Organization
%            addressline={Address One}, 
%            city={City One},
%            postcode={00000}, 
%            state={State One},
%            country={Country One}}

% \author[inst2]{Author Two}
% \author[inst1,inst2]{Author Three}

% \affiliation[inst2]{organization={Department Two},%Department and Organization
%             addressline={Address Two}, 
%             city={City Two},
%             postcode={22222}, 
%             state={State Two},
%             country={Country Two}}

\begin{abstract}In this paper, we propose a new adaptive cross algorithm for computing a low tubal rank approximation of third-order tensors, with less memory and lower computational complexity  than the  truncated tensor SVD (t-SVD). This makes it applicable for decomposing large-scale tensors. We conduct numerical experiments on synthetic and real-world datasets to confirm the efficiency and feasibility of the proposed algorithm. The simulation results show more than one order of magnitude acceleration in the computation of low tubal rank (t-SVD) for large-scale tensors. An application to pedestrian attribute recognition is also presented.
\end{abstract}

%%Graphical abstract
%\begin{graphicalabstract}
%\includegraphics{grabs}
%\end{graphicalabstract}

%%Research highlights
% \begin{highlights}
% \item Research highlight 1
% \item Research highlight 2
% \end{highlights}

\begin{keyword}
%% keywords here, in the form: keyword \sep keyword
Cross tensor approximation, tensor SVD, tubal product
%% PACS codes here, in the form: \PACS code \sep code
%\PACS 0000 \sep 1111
%% MSC codes here, in the form: \MSC code \sep code
%% or \MSC[2008] code \sep code (2000 is the default)
\MSC 15A69 \sep 46N40 \sep 15A23
\end{keyword}

\end{frontmatter}

%% \linenumbers

%% main text
\section{Introduction}
{Tensors} are high-dimensional generalizations of matrices and vectors. Contrary to the rank of matrices, the rank of tensors is not well understood and has to be defined and determined. Different types of tensor decompositions can have different rank definitions such as Tensor Train (TT) \cite{oseledets2011tensor}, Tucker decomposition \cite{tucker1964extension} and its special case, i.e. Higher Order SVD (HOSVD) \cite{de2000multilinear},  CANDECOMP/PARAFAC decomposition (CPD) \cite{hitchcock1928multiple,hitchcock1927expression}, Block Term decomposition \cite{de2008decompositionsI}, Tensor Train/Tensor Ring (TT-TR) decomposition \cite{oseledets2011tensor,zhao2016tensor,espig2012note}, tubal SVD (t-SVD) \cite{kilmer2011factorization}. {The t-SVD factorizes a tensor into three tensors, two orthogonal tensors and one f-diagonal tensor (to be discussed in Section \ref{Sec:tSVD}). Like the SVD for matrices, the truncation version of the t-SVD provides the best tubal rank approximation for every unitary invariant tensor norm. The t-SVD has been successfully applied in deep learning \cite{newman2018stable,newman2019step}, tensor completion \cite{zhang2016exact,zhang2014novel}, image reconstruction \cite{soltani2016tensor} and tensor compression \cite{kilmer2019tensor}.}

Decomposing big data tensors into the t-SVD format is a challenging task, especially when the data is extremely massive and we can not view the entire data tensor. {The} cross, skeleton, or CUR approximation is a useful paradigm widely used for fast low-rank matrix approximation. {Achieving a higher compression ratio, and problems with data interpretation are other motivations to use the cross approximation methods. The main feature of the cross algorithms that makes them effective for managing very large-scale data tensors is their ability to use less memory and have lower computational complexity. When it comes to the higher compression capacity, for instance, the cross matrix approximation provides sparse factor matrices, whereas the SVD of sparse matrices fails to do so, resulting in a more compact data structure. It is also known that the cross approximations can provide more interpretable approximations, we refer to \cite{mahoney2011randomized} for more details.} 

Due to the mentioned motivations, the cross matrix approximation \cite{goreinov1997theory,goreinov2010find} has been generalized to different types of tensor decompositions such as the TT-Cross \cite{oseledets2010tt}, Cross-3D \cite{oseledets2008tucker}, FSTD \cite{caiafa2010generalizing}, and tubal Cross \cite{tarzanagh2018fast}. The cross matrix approximation is generalized to
the tensor case based on the tubal product (t-product) in \cite{tarzanagh2018fast} where some individual lateral
and horizontal slices are selected and based on them a low tubal rank approximation is
computed. The main drawback of this approach is its dependency on the tubal rank estimation, which may be a difficult task in real-world applications. To tackle this problem, we propose to
generalize the adaptive cross matrix approximation \cite{bebendorf2000approximation,bebendorf2006accelerating,zhao2005adaptive} to tensors based on
the t-product. The idea is to select one actual lateral slice and one actual horizontal slice
at each iteration and adaptively check the tubal rank of the tensor.

The generalization of the adaptive cross matrix approximation to tensors based on the t-product is an interesting problem and in this paper, we discuss how to perform it properly. The novelties done in this work include:

\begin{itemize}
    \item A new adaptive tubal tensor approximation algorithm, which estimates the tubal rank and  compute the low tubal rank approximation. The proposed algorithm does not need to use the whole data tensor and at each iteration works only on a part of the horizontal and lateral slices. This facilitates handling large-scale tensors.
    
    \item Presenting an application in pedestrian attribute recognition.
\end{itemize}

The rest of the paper is structured as follows. The basic definitions are given in Section \ref{Sec:prelim}. The t-SVD model is introduced in Section \ref{Sec:tSVD}. The cross matrix approximation and its adaptive version are discussed in Section \ref{Sec:MACA}. Section \ref{Sec:PTACA}, shows how to generalize the adaptive cross approximation to the tensor case based on the t-product. We compare the computational complexity of the algorithms in Section \ref{sec:compcomp}. The experimental results are presented in Section \ref{Sec:Exper} and Section \ref{Sec:Con} concludes the paper and presents potential future directions.   

\section{Preliminaries}\label{Sec:prelim}
The key notations and concepts used in the rest of the paper are introduced in this section. A tensor, a matrix and a vector are denoted by {an underlined} bold capital case letter, a bold capital case letter and a bold lower case letter, respectively. Slices are subtensors generated with fixed all but two modes. Our work is for real-valued third-order tensors but generalization to complex higher order tensor is also straightforward. For a third-order tensor, $\underline{\X},$ the three types of slices $\underline{\X}(:,:,k),\,\underline{\X}(:,j,:),\,\underline{\X}(i,:,:)$ are called frontal, lateral and horizontal slices. For a third-order tensor $\underline{\X},$ the three types fibers $\underline{\X}(:j,k),\,\underline{\X}(i,:,k),\,\underline{\X}(i,j,:)$ are called columns, rows and tubes. The notation ``${\rm conj}$'' means the complex conjugate of all elements (complex numbers) of a matrix. {The notations $\X_{(:,-j)}$ and $\X_{(-i,:)}$ are used to denote new sub-matrices of the matrix $\X$ with the $j$-column and the $i$-th row removed.} The Frobenius norm of tensors/matrices is denoted by $\|.\|_F$ and the Euclidean norm of a vector is shown by $\|.\|_2$. The notation $|.|$ stands for the absolute value of a real number. We need the subsequent definitions to introduce the tensor SVD (t-SVD) model.

\begin{defn} ({t-product})
Let $\underline{\mathbf X}\in\mathbb{R}^{I_1\times I_2\times I_3}$ and $\underline{\mathbf Y}\in\mathbb{R}^{I_2\times I_4\times I_3}$, the t-product $\underline{\mathbf X}*\underline{\mathbf Y}\in\mathbb{R}^{I_1\times I_4\times I_3}$ is defined as follows
\begin{equation}\label{TPROD}
\underline{\mathbf C} = \underline{\mathbf X} * \underline{\mathbf Y} = {\rm fold}\left( {{\rm circ}\left( \underline{\mathbf X} \right){\rm unfold}\left( \underline{\mathbf Y} \right)} \right),
\end{equation}
where 
\[
{\rm circ} \left(\underline{\mathbf X}\right)
=
\begin{bmatrix}
{\mathbf X}^{(1)} &{\mathbf X}^{(I_3)} & \cdots & {\mathbf X}^{(2)}\\
{\mathbf X}^{(2)} &{\mathbf X}^{(1)} & \cdots & {\mathbf X}^{(3)}\\
 \vdots & \vdots & \ddots &  \vdots \\
 {\mathbf X}^{(I_3)} & {\mathbf X}^{(I_3-1)} & \cdots & {\mathbf X}^{(1)}
\end{bmatrix},
\]
and 
\[
{\rm unfold}(\underline{\mathbf Y})=
\begin{bmatrix}
{\mathbf Y}^{(1)}\\
{\mathbf Y}^{(2)}\\
\vdots\\
{\mathbf Y}^{(I_3)}
\end{bmatrix},\hspace*{.5cm}
\underline{\mathbf Y}={\rm fold} \left({\rm unfold}\left(\underline{\mathbf Y}\right)\right).
\]
Here, ${\bf X}^{(i)}=\cX(:,:,i)$ and ${\bf Y}^{(i)}=\cY(:,:,i)$ for $i=1,2,\ldots,I_3$.
\end{defn}{
We denote by $\widehat{\cX}$, the Fourier transform of $\cX$ along its third mode, which can be computed as $\widehat{\cX}={\rm fft}(\cX,[],3)$.  It is known that the block circulant matrix, ${\rm circ}(\cX)\in\mathbb{R}^{I_1 I_3\times I_2 I_3}$, can be block diagonalized, i.e.,
\begin{eqnarray}
({\bf F}_{I_3}\otimes {\bf I}_{I_1})\,{\rm circ}\,(\cX)({\bf F}^{-1}_{I_3}\otimes {\bf I}_{I_2})=\widehat{\X},
\end{eqnarray}
where ${\bf F}_{I_3}\in\mathbb{R}^{I_3\times I_3}$ is the discrete Fourier transform matrix and $({\bf F}_{I_3}\otimes {\bf I}_{I_1})/\sqrt{I_3}$ is a unitary matrix. Here, the block diagonal matrix $\widehat{\X}$ is
\begin{eqnarray}
\widehat{\X}=\begin{bmatrix}
\widehat{\cX}(:,:,1)    &  & & \\
   & \widehat{\cX}(:,:,2)  & & \\
     &   & \ddots & \\
       &  & &  \widehat{\cX}(:,:,I_3)\\
\end{bmatrix},
\end{eqnarray}
and we have the following important properties \cite{rojo2004some,lu2019tensor}
\begin{eqnarray}\label{conj_1}
\widehat{\cX}(:,:,1)&\in&\mathbb{R}^{I_1\times I_2},\\\label{conj_2}
{\rm conj}(\widehat{\cX}(:,:,i))&=&\widehat{\cX}(:,:,I_3-i+2),
\end{eqnarray}
for $i=2,\ldots,\lceil\frac{I_3+1}{2}\rceil+1$.
The t-product can be equivalently performed in the Fourier domain. Indeed, let $\underline{\C}=\cX*\cY$, then from the definition of the t-product and the fact that the block circulant matrix can be block diagonalized, we have 
\begin{eqnarray}
\nonumber
{\rm unfold}(\underline{\C})&=&{\rm circ}(\cX)\,{\rm unfold(\cY)}\\
\nonumber
&=&({\bf F}_{I_3}^{-1}\otimes{\bf I}_{I_1})(({\bf F}_{I_3}\otimes {\bf I}_{I_1})\,{\rm circ}\,(\cX)({\bf F}^{-1}_{I_3}\otimes {\bf I}_{I_2}))\\\label{eq_a}
&&(({\bf F}^{-1}_{I_3}\otimes {\bf I}_{I_2})\,{\rm unfold}(\cY))\\
\nonumber
&=&({\bf F}_{I_3}\otimes{\bf I}_{I_1})\,\widehat{\X}\,{\rm unfold}(\widehat{\cY}),
\end{eqnarray}
where $\widehat{\cY}={\rm fft}(\cY,[],3)$. If we multiply both sides of \eqref{eq_a} from the left-hand side with $({\bf F}_{I_3}\otimes{\bf I}_{I_1})$, we get ${\rm unfold}(\widehat{\underline{\bf C}})=\widehat{\X}\,{\rm unfold}(\widehat{\underline{\bf Y}})$, where $\widehat{\underline{\bf C}}={\rm fft}(\underline{\bf C},[],3)$. This means that $\widehat{\underline{\mathbf C}}\left( {:,:,i} \right) = \widehat{\underline{\mathbf X}}\left( {:,:,i} \right)\,\widehat{\underline{\mathbf Y}}\left( {:,:,i} \right)$. So, it suffices to transform two given tensors into the Fourier domain and multiply their frontal slices. Then, the resulting tensor in the Fourier domain returned back to the original space via the inverse FFT. Note that due to the equations in \eqref{conj_1}-\eqref{conj_2}, half of the computations are reduced. 
This procedure 
is summarized in Algorithm \ref{ALG:TSVDP}. 
}

\begin{defn} ({Transpose})
The transpose of a tensor $\underline{\mathbf X}\in\mathbb{R}^{I_1\times I_2\times I_3}$ is denoted by $\underline{\mathbf X}^{T}\in\mathbb{R}^{I_2\times I_1\times I_3}$ produced by applying the transpose to all frontal slices of the tensor $\underline{\mathbf X}$ and reversing the order of the second untill the last transposed frontal slices.
\end{defn}

\begin{defn} ({Identity tensor})
Identity tensor $\underline{\mathbf I}\in\mathbb{R}^{I_1\times I_1\times I_3}$ is a tensor whose first frontal slice is an identity matrix of size $I_1\times I_1$ and all other frontal slices are zero. It is easy to show $\underline{\I}*\underline{\X}=\underline{\X}$ and $\underline{\X}*\underline{\I} =\underline{\X}$ for all tensors of conforming sizes.
\end{defn}
\begin{defn} ({Orthogonal tensor})
A tensor $\underline{\mathbf X}\in\mathbb{R}^{I_1\times I_1\times I_3}$ is orthogonal (under t-product operator) if ${\underline{\mathbf X}^T} * \underline{\mathbf X} = \underline{\mathbf X} * {\underline{\mathbf X}^ T} = \underline{\mathbf I}$.
\end{defn}

\begin{defn} ({f-diagonal tensor})
If all frontal slices of a tensor are diagonal then the tensor is called an f-diagonal tensor.
\end{defn}

\begin{defn}
(Inverse of a tensor) The inverse of a tensor $\cX\in\mathbb{R}^{I_1\times I_1\times I_3}$ is denoted by $\cX^{-1}\in\mathbb{R}^{I_1\times I_1\times I_3}$ is a unique tensor satisfying 
$
\cX*\cX^{-1}=\cX^{-1}*\cX=\underline{\bf I},
$
where $\underline{\bf I}\in\mathbb{R}^{I_1\times I_1\times I_3}$ is the identity tensor. The inverse of a tensor can also be computed in the Fourier domain and described in Algorithm \ref{ALG:Moore-Penrose}. The MATLAB command ``inv'' in Line 3 computes the inverse of a matrix. The Moore–Penrose (MP) inverse of a tensor $\cX\in\mathbb{R}^{I_1\times I_2\times I_3}$ is denoted by $\cX^{\dag}\in\mathbb{R}^{I_2\times I_1\times I_3}$ and can be computed by Algorithm \ref{ALG:Moore-Penrose} where ``inv'' is replaced with the MATLAB function ``pinv''. Here, ``pinv'' stands for the MP inverse of a matrix.
\end{defn}

\RestyleAlgo{ruled}
\LinesNumbered
\begin{algorithm}
\SetKwInOut{Input}{Input}
\SetKwInOut{Output}{Output}\Input{Two data tensors $\underline{\mathbf X} \in {\mathbb{R}^{{I_1} \times {I_2} \times {I_3}}},\,\,\underline{\mathbf Y} \in {\mathbb{R}^{{I_2} \times {I_4} \times {I_3}}}$} 
\Output{t-product $\underline{\mathbf C} = \underline{\mathbf X} * \underline{\mathbf Y}\in\mathbb{R}^{I_1\times I_4\times I_3}$}
\caption{Fast t-product of two tensors \cite{kilmer2011factorization,lu2019tensor}}\label{ALG:TSVDP}
      {
      $\widehat{\underline{\mathbf X}} = {\rm fft}\left( {\underline{\mathbf X},[],3} \right)$;\\
      $\widehat{\underline{\mathbf Y}} = {\rm fft}\left( {\underline{\mathbf Y},[],3} \right)$;\\
\For{$i=1,2,\ldots,\lceil \frac{I_3+1}{2}\rceil$}
{                        
$\widehat{\underline{\mathbf C}}\left( {:,:,i} \right) = \widehat{\underline{\mathbf X}}\left( {:,:,i} \right)\,\widehat{\underline{\mathbf Y}}\left( {:,:,i} \right)$;\\
}
\For{$i=\lceil\frac{I_3+1}{2}\rceil+1,\ldots,I_3$}{
$\widehat{\underline{\mathbf C}}\left( {:,:,i} \right)={\rm conj}(\widehat{\underline{\mathbf C}}\left( {:,:,I_3-i+2} \right))$;
}
$\underline{\mathbf C} = {\rm ifft}\left( {\widehat{\underline{\mathbf C}},[],3} \right)$;   
       	}       	
\end{algorithm}

\RestyleAlgo{ruled}
\LinesNumbered
\begin{algorithm}
\SetKwInOut{Input}{Input}
\SetKwInOut{Output}{Output}\Input{The data tensor $\underline{\mathbf X} \in {\mathbb{R}^{{I_1} \times {I_1} \times {I_3}}}$} 
\Output{Tensor inverse $\underline{\mathbf X}^{-1}\in\mathbb{R}^{I_1\times I_1\times I_3}$}
\caption{Fast inverse computation of the tensor $\underline{\bf X}$}\label{ALG:Moore-Penrose}
      {
      $\widehat{\underline{\mathbf X}} = {\rm fft}\left( {\underline{\mathbf X},[],3} \right)$;\\
\For{$i=1,2,\ldots,\lceil \frac{I_3+1}{2}\rceil$}
{                        
$\widehat{\underline{\mathbf C}}\left( {:,:,i} \right) = {\rm inv}\,(\widehat{\X}(:,:,i))$;\\
}
\For{$i=\lceil\frac{I_3+1}{2}\rceil+1,\ldots,I_3$}{
$\widehat{\underline{\mathbf C}}\left( {:,:,i} \right)={\rm conj}(\widehat{\underline{\mathbf C}}\left( {:,:,I_3-i+2} \right))$;
}
$\underline{\mathbf X}^{-1} = {\rm ifft}\left( {\widehat{\underline{\mathbf C}},[],3} \right)$;   
       	}       	
\end{algorithm}

\section{Tensor SVD (t-SVD)}\label{Sec:tSVD}
The tensor SVD (t-SVD) represents a tensor as the t-product of three tensors. The first and last tensors are orthogonal, while the middle tensor is an f-diagonal tensor. To be more precise, let $\underline{\bf X}\in\mathbb{R}^{I_1\times I_2\times I_3}$, then the t-SVD of the tensor $\underline{\bf X},$ is $\underline{\bf X} = \underline{\bf U} * \underline{\bf S}* {\underline{\bf V}^T},$ where $\underline{\bf U}\in\mathbb{R}^{I_1\times R\times I_3}$, and $\underline{\bf V}\in\mathbb{R}^{I_2\times R\times I_3}$ are orthogonal tensors and the tensor $\underline{\bf S}\in\mathbb{R}^{R\times R\times I_3}$ is f-diagonal \cite{kilmer2011factorization,kilmer2013third}, see Figure \ref{fig:tSVD} for an illustration on the t-SVD and its truncated version. Note that Algorithm \ref{ALG:TQR} only needs the truncated SVD of the $\lceil \frac{I_3+1}{2}\rceil$ first frontal slices. The generalization of the t-SVD to tensors of order higher than three is done in \cite{martin2013order}.
\begin{figure}
\begin{center}
    \includegraphics[width=.7\linewidth]{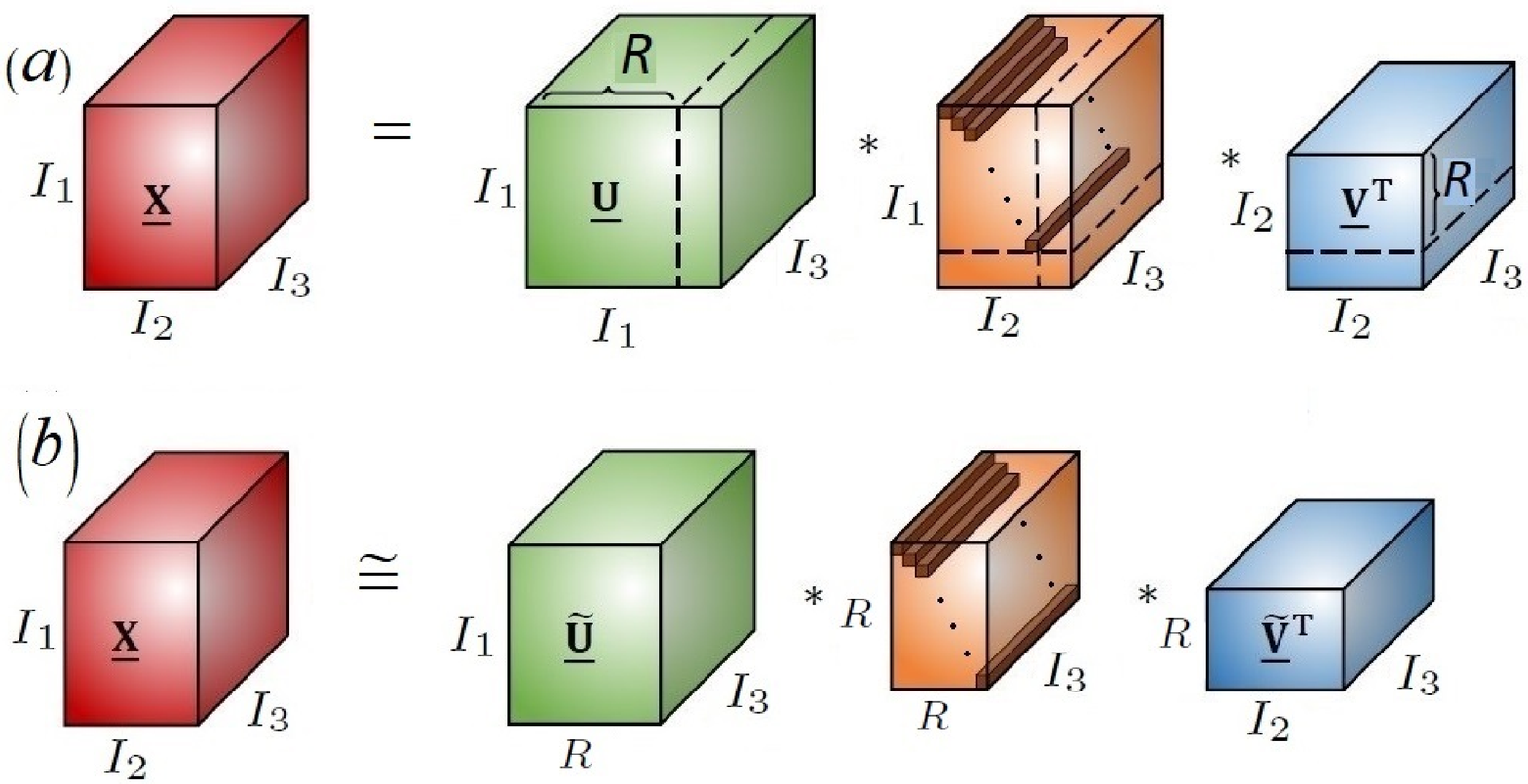}
    \caption{a) Tensor SVD (t-SVD) of the tensor $\cX$, b) The truncated t-SVD of the tensor $\cX$ for the tubal rank $R$ \cite{ahmadi2022efficient}.}\label{fig:tSVD}
    \end{center}
\end{figure}
Other types of classical matrix decompositions such as QR and LU decompositions can be generalized based on the t-product in straightforward ways.
%, for example see Algorithm \ref{ALG:TQR2} for the t-QR computation. 

The computational complexity of Algorithm \ref{ALG:TQR} is dominated by the FFT of all tubes of an input tensor and also the truncated SVD of the frontal slices in the Fourier domain. In the literature, some algorithms have been developed to accelerate these computations. For example, using the idea of randomization, we can replace the classical truncated SVD with more efficient and faster approaches such as {the} randomized SVD \cite{halko2011finding,mahoney2011randomized} or cross matrix approximation. Although this idea can somehow solve the mentioned computation difficulty of Algorithm \ref{ALG:TQR}, still we need to access all elements of the underlying data tensor. For very big data tensors where viewing the data tensor even once is very prohibitive, it is required to develop  algorithms that only use a part of the data tensor at each iteration. In this paper, we follow this idea and propose an efficient algorithm for the computation of the t-SVD, which uses only a part of lateral and horizontal slices of a tensor at each iteration. This significantly accelerates the computations and in some of our simulations, we have achieved almost two orders of magnitude acceleration, which shows the performance of the proposed algorithm. To the best of our knowledge, this is the first adaptive cross algorithm developed for the computation of the t-SVD.

\begin{figure}
    \includegraphics[width=0.9\linewidth]{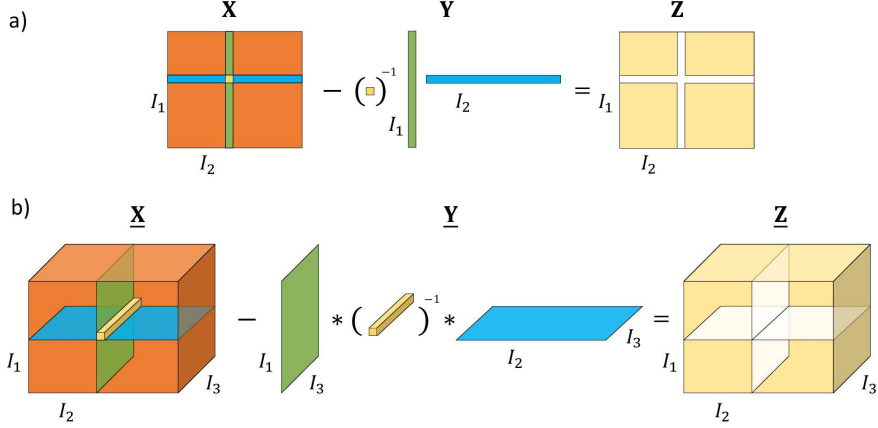}
    \caption{a) One stage of the adaptive cross matrix approximation for rank reduction. The corresponding column and row of the residual matrix ${\bf Z},$ with the same indices as the selected column and row of the original data matrix $\X$ become zeros, i.e. ${\rm rank}({\bf X}-{\bf Y})={\rm rank}({\bf X})-1.$ The rank one matrix approximation ${\bf Y}$ interpolates ${\bf X}$ at the selected column and row. b) One stage of the cross tubal approximation for the tubal rank reduction. The corresponding lateral and horizontal slices of the residual tensor $\underline{\bf Z},$ with the same indices as the selected lateral and horizontal slices of the original data tensor $\underline{\bf X}$ become zeros, ${\rm rank}(\underline{\bf X}-\underline{\bf Y})={\rm rank}(\underline{\bf  X})-1$. The tubal rank one tensor approximation $\cY$ interpolates $\cX$ at the selected lateral and horizontal slices.}\label{fig:CP}
\end{figure}

%The classical matrix decompositions such as QR, LU and SVD can be straightforwardly generalized to tenors using the t-product. Given a tensor $\X\in\mathbb{R}^{I_1\times I_2\times I_3},$ the tensor QR (t-QR) represent a tensor in the form $\X=\Q*\R$ and can be computed through Algorithm \ref{ALG:TQR}. By a slight modification of Algorithm  \ref{ALG:TQR}, the tensor LU (t-LU) and tensor SVD (t-SVD) can be computed. More precisely, in line 4 of Algorithm  \ref{ALG:TQR}, we replace the LU decomposition and the SVD of frontal slices $\widehat{\X}(:,:,i),\,i=1,2,\ldots,I_3,$ instead of QR decomposition. Please note that Algorithm \ref{ALG:TQR} only needs the thin QR of the first $\lceil \frac{I_3+1}{2}\rceil$ slices while the original t-QR algorithm developed in \cite{kilmer2011factorization,kilmer2013third} involves the QR of all frontal slices. So, it is recommended to utilize this idea and remove the redundant computations. 

\RestyleAlgo{ruled}
\LinesNumbered
\begin{algorithm}
\SetKwInOut{Input}{Input}
\SetKwInOut{Output}{Output}\Input{The data tensor $\underline{\mathbf X} \in {\mathbb{R}^{{I_1} \times {I_2} \times {I_3}}}$ and a target tubal rank $R$} 
\Output{The truncated t-SVD of the tensor $\cX$}
\caption{The truncated t-SVD decomposition of the tensor $\underline{\bf X}$}\label{ALG:TQR}
      {
      $\widehat{\underline{\mathbf X}} = {\rm fft}\left( {\underline{\mathbf X},[],3} \right)$;\\
\For{$i=1,2,\ldots,\lceil \frac{I_3+1}{2}\rceil$}
{                        
$[\widehat{\underline{\mathbf U}}\left( {:,:,i} \right),\widehat{\underline{\bf  S}}(:,:,i),\widehat{\V}(:,:,i)] = {\rm svds}\,(\widehat{\X}(:,:,i),R)$;\\
}
\For{$i=\lceil\frac{I_3+1}{2}\rceil+1,\ldots,I_3$}{
$\widehat{\underline{\mathbf U}}\left( {:,:,i} \right)={\rm conj}(\widehat{\underline{\mathbf U}}\left( {:,:,I_3-i+2} \right))$;\\
$\widehat{\underline{\mathbf S}}\left( {:,:,i} \right)=\widehat{\underline{\mathbf S}}\left( {:,:,I_3-i+2} \right)$;\\
$\widehat{\underline{\mathbf V}}\left( {:,:,i} \right)={\rm conj}(\widehat{\underline{\mathbf V}}\left( {:,:,I_3-i+2} \right))$;
}
$\underline{\mathbf U}_R= {\rm ifft}\left( {\widehat{\underline{\mathbf U}},[],3} \right)$;
$\underline{\mathbf S}_R= {\rm ifft}\left( {\widehat{\underline{\mathbf S}},[],3} \right)$; 
$\underline{\mathbf V}_R= {\rm ifft}\left( {\widehat{\underline{\mathbf V}},[],3} \right)$
       	}       	
\end{algorithm}

% \RestyleAlgo{ruled}
% \LinesNumbered
% \begin{algorithm}
% \SetKwInOut{Input}{Input}
% \SetKwInOut{Output}{Output}\Input{The data tensor $\underline{\mathbf X} \in {\mathbb{R}^{{I_1} \times {I_2} \times {I_3}}}$} 
% \Output{The t-QR decomposition of the tensor $\cX$}
% \caption{Fast t-QR decomposition of the tensor $\underline{\bf X}$}\label{ALG:TQR2}
%       {
%       $\widehat{\underline{\mathbf X}} = {\rm fft}\left( {\underline{\mathbf X},[],3} \right)$;\\
% \For{$i=1,2,\ldots,\lceil \frac{I_3+1}{2}\rceil$}
% {                        
% $[\widehat{\underline{\mathbf Q}}\left( {:,:,i} \right),\widehat{\R}(:,:,i)] = {\rm qr}\,(\widehat{\X}(:,:,i),0)$;\\
% }
% \For{$i=\lceil\frac{I_3+1}{2}\rceil+1,\ldots,I_3$}{
% $\widehat{\underline{\mathbf Q}}\left( {:,:,i} \right)={\rm conj}(\widehat{\underline{\mathbf Q}}\left( {:,:,I_3-i+2} \right))$;\\
% $\widehat{\underline{\mathbf R}}\left( {:,:,i} \right)={\rm conj}(\widehat{\underline{\mathbf R}}\left( {:,:,I_3-i+2} \right))$;
% }
% $\underline{\mathbf Q}= {\rm ifft}\left( {\widehat{\underline{\mathbf Q}},[],3} \right)$;\\
% $\underline{\mathbf R}= {\rm ifft}\left( {\widehat{\underline{\mathbf R}},[],3} \right)$; 
%        	}       	
% \end{algorithm}

\section{Matrix cross approximation and its adaptive version}\label{Sec:MACA}
The cross matrix approximation was first proposed in \cite{goreinov1997theory} for fast low-rank approximation of matrices. It provides a low-rank matrix approximation based on some actual columns and rows of the original matrix. It has been shown that a cross approximation with the maximum volume of the intersection matrix leads to close to optimal approximation \cite{goreinov2010find}. The adaptive cross approximation or Cross2D algorithm \cite{savostyanov2006polilinear,tyrtyshnikov2000incomplete,bebendorf2000approximation,bebendorf2006accelerating,zhao2005adaptive} sequentially selects a column and a row of the original data matrix and based on them, computes a rank-1 matrix scaled by the intersection element, as stated in the following theorem, which indeed is the Gaussian elimination process. 

\begin{thm}\label{Thm_1} (Rank-1 deflation)
Let $\X\in\mathbb{R}^{I_1\times I_2}$ be a given matrix, and we select the $i$-th row and the $j$-th column with the nonzero intersection element $\X(i,j)$. Then the following residual matrix 
\[
\Y=\X-\frac{1}{\X(i,j)}\X(:,j)\X(i,:),
\]
vanishes at the $i$-th row and $j$-th column, so ${\rm rank}(\Y)={\rm rank}(\X)-1$.
\end{thm}

\begin{proof}

It is obvious that the $j$-th column and $i$-th row of $\Y$ are zeros
\begin{eqnarray}
\Y(:,j)=\X(:,j)-\frac{1}{\X(i,j)}\X(:,j)\X(i,j)=0,\\
\Y(i,:)=\X(i,:)-\frac{1}{\X(i,j)}\X(i,j)\X(i,:)=0.
\end{eqnarray}

To prove Theorem \ref{Thm_1}, we consider two cases: 
\begin{enumerate}
    \item If $\X$ is of full-rank, that is, either $\X(:,j)\notin\,{\rm range}\,(\X_{(:,-j)})$ or $\X(i,:)\notin\,{\rm range}\,(\X_{(-i,:)})$, then it is straightforward that $\Y$ has smaller rank than $\X$.
    
    \item Otherwise, we consider the case $\X$ is rank-deficient and $\X(:,j)$ and $\X(i,:)$ are non zero vectors that $\X(i,:)\in {\rm range}\,(\X_{(-i,:)})$ and $\X(:,j)\in {\rm range}\,(\X_{(:,-j)})$.
    
    Without loss of generality, we assume that $\X(:,j)$ and $\X(i,:)$ are the last column and the last row of the matrix $\X$, i.e. $i=I_1,\,j=I_2$, (see illustration in Figure \ref{fig:thm}) and rank($\X$) $< \min(I_1,I_2)$.
    
    Since $\X(i,:)\in {\rm range}\,(\X_{(-i,:)})$ and $\X(:,j)\in {\rm range}\,(\X_{(:,-j)})$, there exist linear combinations such that 
\begin{eqnarray}
    \X(:,j) = \begin{bmatrix}
        {\bf a}\\{c}
    \end{bmatrix}=\begin{bmatrix}
        {\bf Z}\\{\bf b}^T 
    \end{bmatrix}{\boldsymbol \alpha},\quad\quad
       \X(i,:)^T = \begin{bmatrix}
        {\bf b}\\{ c}
    \end{bmatrix}=\begin{bmatrix}
        {\bf Z}^T\\{\bf a}^T 
    \end{bmatrix}{\boldsymbol \beta},
\end{eqnarray}
where ${\boldsymbol \alpha}\neq 0,\,{\boldsymbol \beta}\neq 0$.  This gives $c = {\boldsymbol \beta}^T {\bf Z} {\boldsymbol \alpha}$.

\begin{figure}
\centering
    \includegraphics[width=0.3\linewidth]{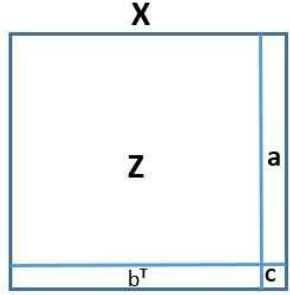}
    \caption{Partitioning the matrix $\bf X$ for the proof of Theorem \ref{Thm_1}.}\label{fig:thm}
\end{figure}
The rank-1 matrix deflation yields 
\begin{eqnarray}
\Y=\X-\frac{1}{c}\begin{bmatrix}
    {\bf a}\\ c
\end{bmatrix}
\begin{bmatrix}
    {\bf b}^T & c
\end{bmatrix}=\begin{bmatrix}
    {\bf Z}-\frac{1}{c}{\bf a}{\bf b}^T & {\bf 0}\\
    {\bf 0} & 0
\end{bmatrix}.
\end{eqnarray}
where the top-left submatrix of $\Y$ has rank-1 reduction from $\bf Z$ 
\begin{eqnarray}\label{EQ_a}
{\bf W}={\bf Z}-\frac{1}{c}{\bf a}{\bf b}^T={\bf Z}-\frac{1}{c}{\bf Z}{\boldsymbol \alpha} \, {\boldsymbol \beta}^T{\bf Z}.
\end{eqnarray}
We next substitute ${\bf Z}$ by its singular value decomposition ${\bf Z}={\bf U}{\bf S}{\bf V}^T$, where $\bf S$ is a diagonal matrix of positive singular values of $\bf Z$ and consider
\begin{eqnarray}\label{EQn}
{\bf W}={\bf U}\left(\underbrace{{\bf S}-\frac{1}{c}
{\bf S}({\bf V}^T{\bf {\boldsymbol \alpha}})({\boldsymbol \beta}^T\bf U) {\bf S}}_{\bf K}\right){\bf V}^T.    
\end{eqnarray}{
Assume ${\bf d}={\bf V}^T{\boldsymbol\alpha}$ and ${\bf e}={\bf U}^T{\boldsymbol \beta }$, then from \eqref{EQn}, we have 
\begin{eqnarray}
 {\bf K}={\bf S}-\frac{{\bf S}{\bf d}{\bf e}^T{\bf S}}{{\bf e}^T{\bf S}{\bf d}}.   
\end{eqnarray}
It is straightforward to see that 
\begin{eqnarray}
{\bf S}^{-1/2}{\bf K}{\bf S}^{-1/2}={\bf I}-\frac{{\bf u}{\bf v}^T}{{\bf v}^T{\bf u}},   
\end{eqnarray}
where ${\bf u}={\bf S}^{1/2}{\bf d}$ and ${\bf v}={\bf S}^{1/2}{\bf e}$. Now, it is readily seen that ${\bf K}$ has a zero singular value, i.e., its rank is reduced from the rank of $\bf S$ by 1. Since the multiplication with orthogonal matrices does not change the matrix rank, the proof of the theorem is completed.}

\begin{rem}{
An alternative proof for Theorem \ref{Thm_1} adopts the following fact proved in \cite{chu1995rank}.}
{
Given $\X\in\mathbb{R}^{m\times n}$ and assume $\U\in\mathbb{R}^{m\times k},\,\R\in \mathbb{R}^{k\times k}$, and $\V\in\mathbb{R}^{n\times k}$. Then 
\begin{eqnarray}\label{et}
{\rm rank}(\X - \U\R^{-1}\V^T) = {\rm rank}(\X) - {\rm rank}(\U\R^{-1}\V^T), 
\end{eqnarray}
if and only if there exist $\A\in \mathbb{R}^{n\times k}$ and $\B\in\mathbb{R}^{m\times k}$ such that $\U = \X\A$, $\V = \X^T\B$, and $\R = \B^T\X\A$.}
{
If we define $\A={\bf e}_j\in\mathbb{R}^{n},\,\,\B={\bf  e}_i\in\mathbb{R}^{m}$ as the $j$-th and $i$-th standard unit vectors\footnote{For the standard unit vector $e_i$, the $i$th element is 1 and the rest are zero.}. Then, we have $\U=\X\A=\X(:,j),\,\,\V=\X^T\B=\X(i,:)^T$ and $\R^{-1}=({\bf e}_i^T{\bf X}{\bf e}_j)^{-1}=\frac{1}{{\bf X}(i,j)}$. Now \ref{et}, demonstrates that 
\begin{eqnarray*}
{\rm rank}(\X - \frac{1}{c}\X(:,j)\X(i,:)) &=& {\rm rank}(\X) - {\rm rank}(\frac{1}{c}\X(:,j)\X(i,:)) \\
&=&{\rm rank}(\X) -1.
\end{eqnarray*}}

{
So, this completes the proof.}
\end{rem}

%Since $({\bf U}^T {\boldsymbol \beta})^T{\bf K} ({{\bf V}}^T{\boldsymbol \alpha})=0$ and ${\bf U}^T {\boldsymbol 
%\beta}\neq 0$, ${\bf V}^T{\boldsymbol \alpha}\neq 0$, ${\bf K}$ has a zero singular value, i.e., its rank is reduced from rank of $\bf S$ by 1. This completes 
 %the proof.
\end{enumerate}

\end{proof}
In view of Theorem \ref{Thm_1}, we see that the corresponding column and row of the residual matrix $\Y,$ with the same indices as the selected column/row of the original data matrix $\X$ become zero, which means that this approximation interpolates the original data matrix at the mentioned indices and reduces its rank by one order, see Figure \ref{fig:CP} (a) for a graphical illustration of this approach. This procedure is repeated by selecting a new column and a new row of the residual matrix, so we can sequentially reduce the matrix rank and interpolate the original matrix at the new columns/rows. The adaptive cross matrix approximation method is summarized in Algorithm \ref{Cross_Alg}. Clearly, the breakdown can happen in Algorithm \ref{Cross_Alg} if the denominator ${\bf u}_k(i_k),$ becomes zero. If we deal with such a case, we should select a new index $i_k$ for which the ${\bf u}_k(i_k)$ is not zero. For example one can select randomly a new index, which has not been chosen previously. 

\RestyleAlgo{ruled}
\LinesNumbered
\begin{algorithm}
\SetKwInOut{Input}{Input}
\SetKwInOut{Output}{Output}\Input{A data matrix ${\mathbf X} \in {\mathbb{R}^{{I_1} \times {I_2}}}$, an approximation error $\epsilon$}
\Output{Low rank matrix approximation ${\mathbf X}={\U}{\V}$}
${\bf U}=0,\,{\bf V}=0,\,\mu=0,\,r_0=0,$ $j_1$ is a random column index
\caption{Adaptive cross approximation algorithm
(ACA) \cite{savostyanov2006polilinear,tyrtyshnikov2000incomplete,bebendorf2000approximation,bebendorf2006accelerating,zhao2005adaptive} }\label{Cross_Alg}
\For{$k=1,2,\ldots,{\rm min}(I_1,I_2)$}
{
${\bf u}_{k}=\E_{k-1}(:,j_k)=\X(:,j_k)-\U\V(:,j_k)$;\\
$i_k={\rm \arg\max}_i |{\bf u}_k(i)|$;\\
${\bf u}_k\leftarrow {\bf u}_k/{\bf u}_k(i_k)$;\\
${\bf v}_k^T={\bf E}_{k-1}(i_k,:)={\bf X}(i_k,:)-{\bf U}(i_k,:){\bf V}$;\\
$j_{k+1}=\arg\max_{j}|{\bf v}_k(j)|$;\\
$\rho^2=\|{\bf u}_k\|_2^2\|{\bf v}_k\|_2^2$\\
$\mu^2\leftarrow\mu^2+\rho^2+2\sum_{j=1}^{k-1}{\bf V}(j,:){\bf v}_k{\bf u}_k^T{\bf U}(:,j)$;\\
${\bf U}\leftarrow [{\bf U},{\bf u}_k],\,{\bf V}\leftarrow [{\bf V};{\bf v}^T_k];$\\
$\,r_k=r_{k-1}+1$;\\
\If{$\rho<\epsilon\mu$}{{\bf Break}}
}       	
\end{algorithm}

\section{Proposed adaptive tensor cross approximation based on the t-product}\label{Sec:PTACA}
In this section, we show how to generalize the adaptive cross matrix approximation to the tensor case based on the t-product. Compared to the matrix case, instead of a column and a row, here we select a lateral slice and a horizontal slice at each iteration but the important question is how to use the intersection tube for scaling the corresponding tubal rank-1 tensor so that the tubal rank of the residual tensor is reduced one in order. We found that the inverse of the intersection tube should be used and this is proved in Theorem \ref{Thm_2}. 
\begin{thm}(Tubal rank-1 deflation)\label{Thm_2}
Let $\cX\in\mathbb{R}^{I_1\times I_2\times I_3}$ be a given data tensor with sampled lateral and horizontal slices as $\cX(:,j,:)$ and $\cX(i,:,:)$ with the nonzero intersection tube $\X(i,j,:)$. Then the residual tensor 
\begin{equation}\label{Residual}
\cY=\cX-\cX(:,j,:)*(\cX(i,j,:))^{-1}*\cX(i,:,:),
\end{equation}
vanishes at its $i$-th lateral slice and $j$-th horizontal slice and ${\rm rank}(\cY)={\rm rank}(\cX)-1$.  
\end{thm}

\RestyleAlgo{ruled}
\LinesNumbered
\begin{algorithm}
\SetKwInOut{Input}{Input}
\SetKwInOut{Output}{Output}\Input{A data tensor $\cX \in {\mathbb{R}^{{I_1} \times {I_2\times I_3} }}$, an approximation error $\epsilon$}
\Output{Low tubal rank tensor approximation $\cX={\cU}*{\cV}$}
$\cU=0,\,\cV=0,\,\mu=0,\,r_0=0,$ $j_1$ is a random lateral slice index
\caption{Proposed adaptive cross tubal tensor approximation algorithm
(ACTA)}\label{CrossTensor_Alg}
\For{$k=1,2,\ldots,{\rm min}(I_1,I_2)$}
{
$\underline{\bf u}_{k}=\cE_{k-1}(:,j_k,:)=\cX(:,j_k,:)-\cU*\cV(:,j_k,:)$;\\
$i_k={\rm \arg\max}_i\,\, ||\underline{\bf u}_k(i,j_k,:)||_2^2$;\\
${\underline{\bf u}}_k\leftarrow {\underline{\bf u}}_k*({\underline{\bf u}}_k(i_k,j_k,:))^{-1}$;\\
${\underline{\bf v}}_k^T=\cE_{k-1}(i_k,:,:)=\cX(i_k,:,:)-\cU(i_k,:,:)*\cV$;\\
$j_{k+1}=\arg\max_{j}\,\,||{\underline{\bf v}}_k(i_k,j,:)||_2^2$;\\
$\rho^2=\|{\underline{\bf u}}_k\|_F^2\|{\underline{\bf v}}_k\|_F^2$;\\
$\mu^2\leftarrow\mu^2+\rho^2+2\|\sum_{j=1}^{k-1}{\underline{\bf v}}(j,:,:)*{\underline{\bf v}}_k*{\underline{\bf u}}_k^T*{\underline{\bf u}}_k(:,j,:)\|_2^2$;\\
$\cU\leftarrow [\cU,{\underline{\bf u}}_k],\,\cV\leftarrow [\cV;{\underline{\bf v}}^T_k];$\\
$r_k=r_{k-1}+1$;\\
\If{$\rho<\epsilon\mu$}{{\bf break}}
}       	
\end{algorithm}

\begin{proof}
To prove Theorem \ref{Thm_2}, we show that the $i$-th row and the $j$-th column of each frontal slice of the residual tensor $\cY$ are zero. To do this, let us consider the $k$-th frontal slice of the residual tensor $\cY$ as $\cY(:,;,k)$. In the Fourier domain, it can be represented as 
\begin{eqnarray}\label{resiFour}
\widehat{\cY}(:,:,k)=\widehat{\cX}(:,:,k)-\frac{1}{\widehat{\cX}(i,j,k)}\widehat{\cX}(:,j,k)\widehat{\cX}(i,:,k).
\end{eqnarray}
In view of Theorem \ref{Thm_1}, this means that the $i$-th row and the $j$-th column of the the $k$-th frontal slice of the tensor $\cY$ in the Fourier domain are zero and its rank is one order lower than the rank of the matrix $\widehat{\bf X}(:,:,k)$. So the $i$-th row and the $j$-th column of the all frontal slices  $\cX(:,:,k),\,k=1,2,\ldots,I_3$ equal to zero. This clearly completes the proof.
\end{proof}
It is not difficult to see that for second-order tensors (matrices), Equation \eqref{Residual} is reduced to the classical matrix cross approximation. At each iteration, we select a lateral slice and a horizontal slice and perform the scaling using the pseudoinverse of the intersection tube. The corresponding scaled tubal rank-1 tensor reduces the tubal rank of the underlying data tensor by one order. It is interesting to note that similar to the matrix case where after each iteration the corresponding selected column and row in the residual matrix become zeros, here the corresponding lateral and horizontal slices of the residual tensor vanish. So, naturally, this approximation interpolates the original tensor at the mentioned slices. This procedure can proceed with the residual tensor to reduce the tubal rank sequentially. The generalized adaptive cross tubal approximation method is outlined in Algorithm \ref{CrossTensor_Alg}. In Line 10 of Algorithm \ref{CrossTensor_Alg}, the new horizontal and lateral slicers are concatenated along the second and first modes, respectively. {The relative error accuracy is used for the stopping criterion as $\rho\leq \epsilon\mu$ according to 
\begin{eqnarray*}
\rho=\|\underline{\bf u}_k*\underline{\bf v}^T_k\|_F&\approx&\|\underline{\bf X}-\underline{\bf U}*\underline{\bf V}\|_F,\\
\mu=\|\underline{\bf U}*\underline{\bf V}\|_F&\approx&\|\underline{\bf X}\|_F,
\end{eqnarray*}
where $\underline{\bf U}=[\underline{\bf u}_1,\ldots,\underline{\bf u}_{k-1}]$ and $\underline{\bf V}=[\underline{\bf v}_1;\ldots;\underline{\bf v}_{k-1}]$. It is necessary to enforce $i_k\neq i_1,i_2,\ldots,i_{k-1}$ and $j_k\neq j_1,j_2,\ldots,j_{k-1},$ which means that each iteration should produce new indices (different from the others).
%Please note that Algorithm \ref{CrossTensor_Alg} is prone to the breakdown in Line 5. If this case happens we ignore the selected index $i_k$ and find a new one randomly which has not chosen yet and check again the breakdown condition.
We also remark that if the size of a frontal/lateral slice is big, one can compress it using the classical cross methods, similar to \cite{oseledets2008tucker} where the cross approximation is used in two stages for the computation of the Tucker decomposition. Although our results so far are for third-order tensors, clearly they can be straightforwardly generalized to tensors of a higher order than three, according to \cite{martin2013order}.}

\section{Computational complexity}\label{sec:compcomp}
The adaptive cross tensor algorithm is efficient as it only works on a lateral slice and a horizontal slice at each iteration. The computational complexity of Algorithm \ref{CrossTensor_Alg} is $\mathcal{O}((I+J)K\log(K))$. 
The computational complexity of the truncated t-SVD for a tensor of the size $I\times J\times K$ is 
$\mathcal{O}(IJK{\rm log}(K))+\mathcal{O}(IJK{\rm min}(I,J))$. Besides, the truncated t-SVD needs to access and process the whole data tensor while the proposed algorithm works only on a part of the lateral slice and  horizontal slices at each iteration. So, it is clearly seen that the proposed Algorithm \ref{CrossTensor_Alg} requires much less memory and computational operation than the t-SVD algorithm. This makes it applicable for decomposing large-scale tensors.

\section{Experimental Results}\label{Sec:Exper}
We have used \textsc{Matlab} and some functions of the toolbox 

\url{https://github.com/canyilu/Tensor-tensor-product-toolbox} 

to implement the proposed algorithm using a laptop computer with 2.60 GHz Intel(R) Core(TM) i7-5600U processor and 8GB memory. {We have used two metrics, {\it relative error} and Peak signal-to-noise ratio (PSNR) to compare the efficiency of the proposed algorithm with the baselines. The relative error is defined as follows
\[
{\rm Relative\,\,error}=\frac{\|\cX-\cU*\underline{\bf S}*\cV^T\|_F}{\|\cX\|_F}.
\]
The PSNR is also defined as
\[{\rm{PSNR = 10lo}}{{\rm{g}}_{{\rm{10}}}}\left( {{{255}^2}/{\rm{MSE}}} \right),
\]
where ${\rm{MSE}} = \left\| {\underline{\bf X} - \widehat{\underline{\bf X}}} \right\|_F^2/{\rm{num}}\left( \underline{\bf X} \right).$
Note ``num'' denotes the number of parameters of a given data tensor.
We mainly consider three examples. In the first example, we examine the algorithms using low-rank random data tensors. In the second example, we have used the functional based tensors. In the last example, we used the images as real-world data with application to the image completion problem.}

\begin{exa}\label{exm1}
In this example we consider a random data tensor $\X\in\mathbb{R}^{N\times N\times N}$ with exact tubal rank $R=30$ for $N=100,200,\ldots,600$. To generate such a tensor, we considered two standard Gaussian tensors and orthonormalize them. Let us denote these orthogonal parts by $\cU\in\mathbb{R}^{N\times R\times N}$ and $\cV\in\mathbb{R}^{N\times R\times N}$. Then, we generate a tensor $\underline{\mathbf S}\in\mathbb{R}^{R\times R\times N}$ with only $R$ nonzero diagonal tubes $\underline{\mathbf S}(i,i,:),\,i=1,2,\ldots,R$ whose elements are also standard Gaussian and build the tensor $\cX=\cU*\underline{\mathbf S}*\cV^T,$ which is used in our simulations. Assume that $R=30$ in our simulations, and set $\epsilon=10^{-8}$ in Algorithm \ref{CrossTensor_Alg}. Then, we apply the proposed algorithm to find the tubal rank and the corresponding low tubal rank approximation. We consider 100 Monte Carlo experiments and report the mean of our results (accuracy and running time). In all our experiments, the proposed approach retrieved the true tubal rank successfully, and this convinced us that it works well for finding the tubal rank of a tensor. Then we used the truncated t-SVD and the randomized t-SVD \cite{tarzanagh2018fast} to compute low tubal rank approximations of the underlying data tensor. The running time of the proposed algorithm, the truncated t-SVD, the randomized t-SVD are compared in Figure \ref{fig1} (right). The numerical results show almost two orders of magnitude speed-up of the proposed approach compared with the truncated t-SVD algorithm, while it is also faster than the randomized t-SVD. The accuracy comparison of the algorithms is also presented in Figure \ref{fig1} (left). This illustrates that the proposed algorithm can provide acceptable results in much less time than the truncated t-SVD algorithm and the randomized t-SVD.  

\begin{figure}
\begin{center}
    \includegraphics[width=0.49\linewidth]{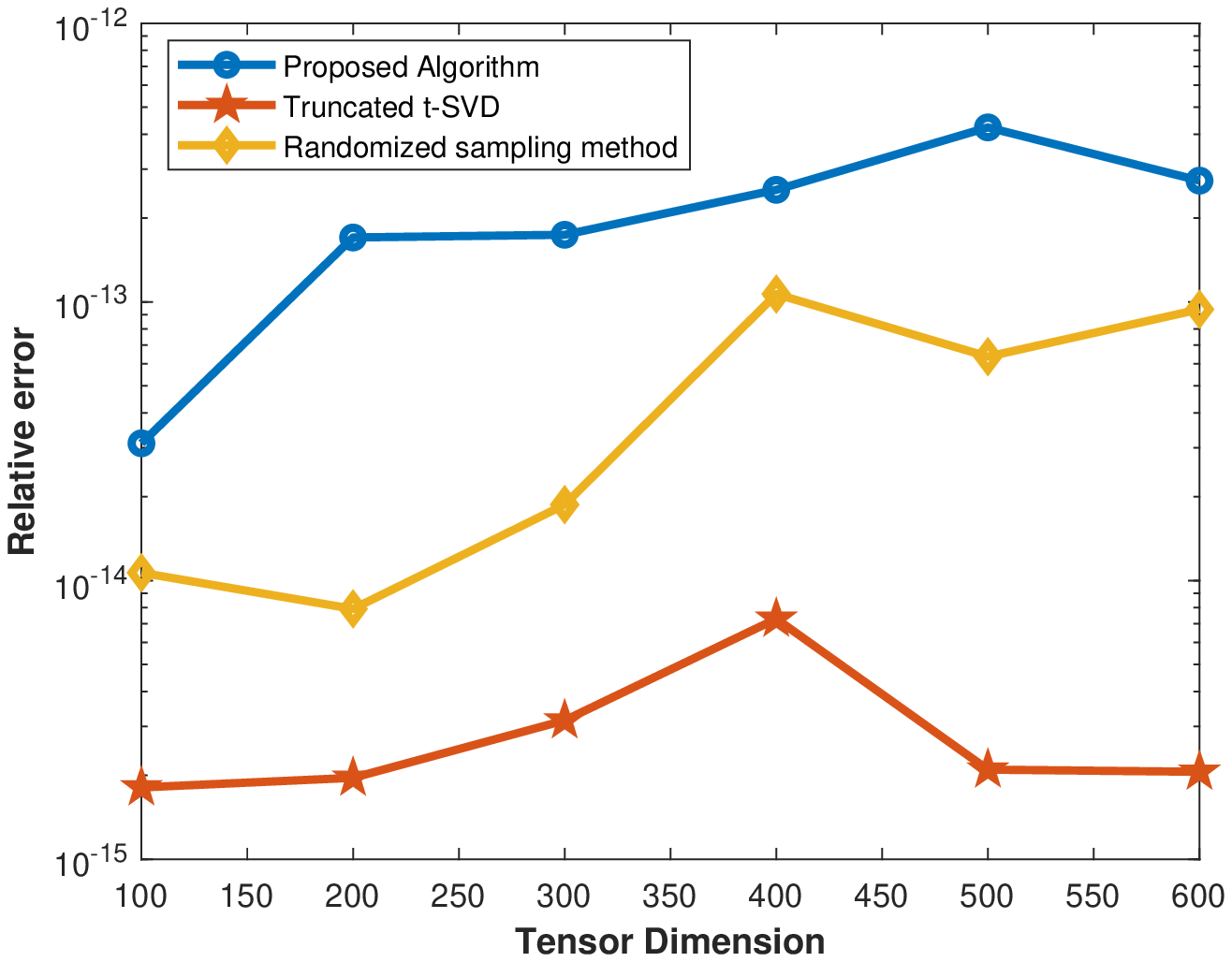}
    \includegraphics[width=0.49\linewidth]{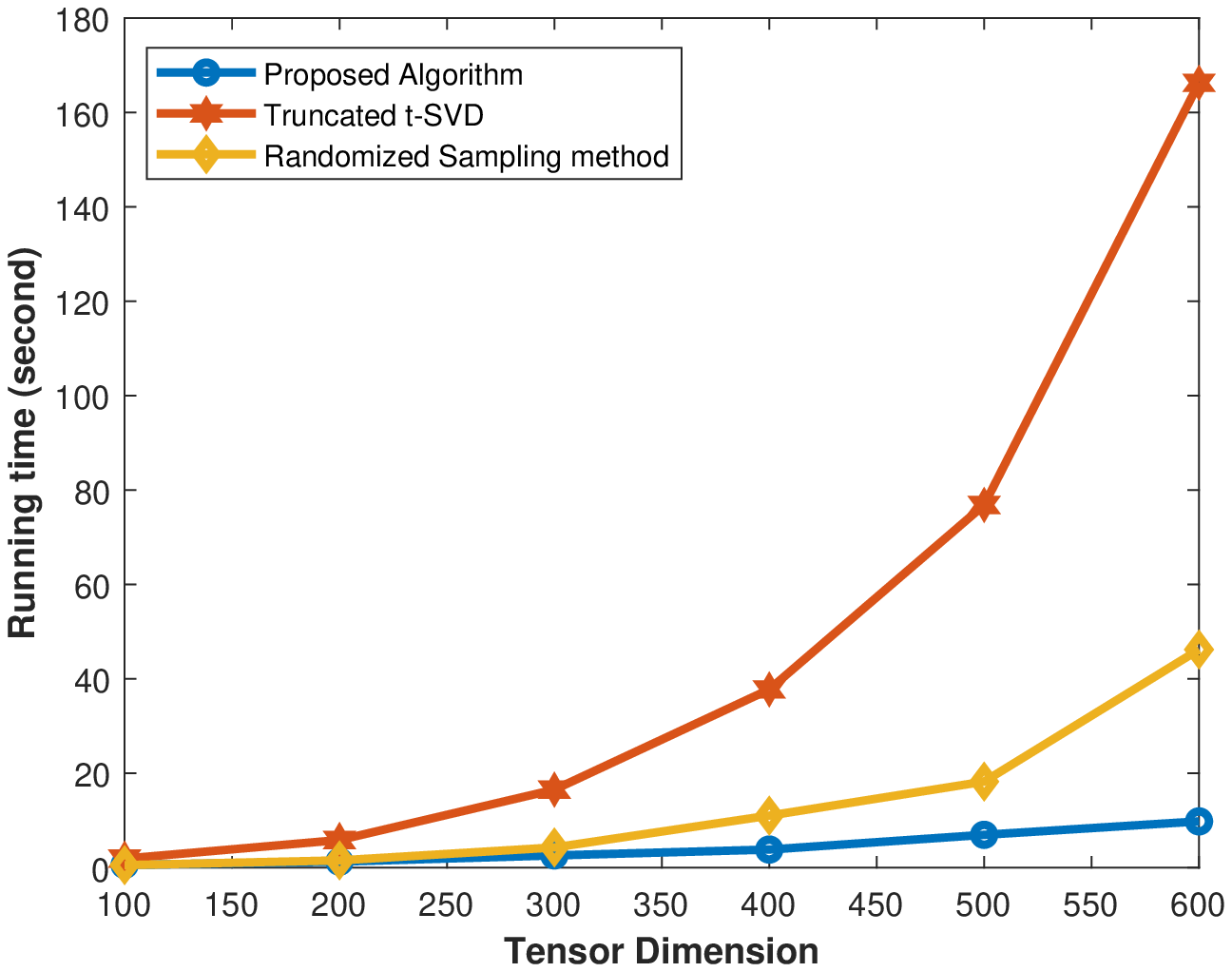}
    \caption{The running time and accuracy comparisons of the proposed algorithm and the truncated t-SVD for Example \ref{exm1}.}\label{fig1}
    \end{center}
\end{figure}

\begin{figure}
\begin{center}
    \includegraphics[width=0.49\linewidth]{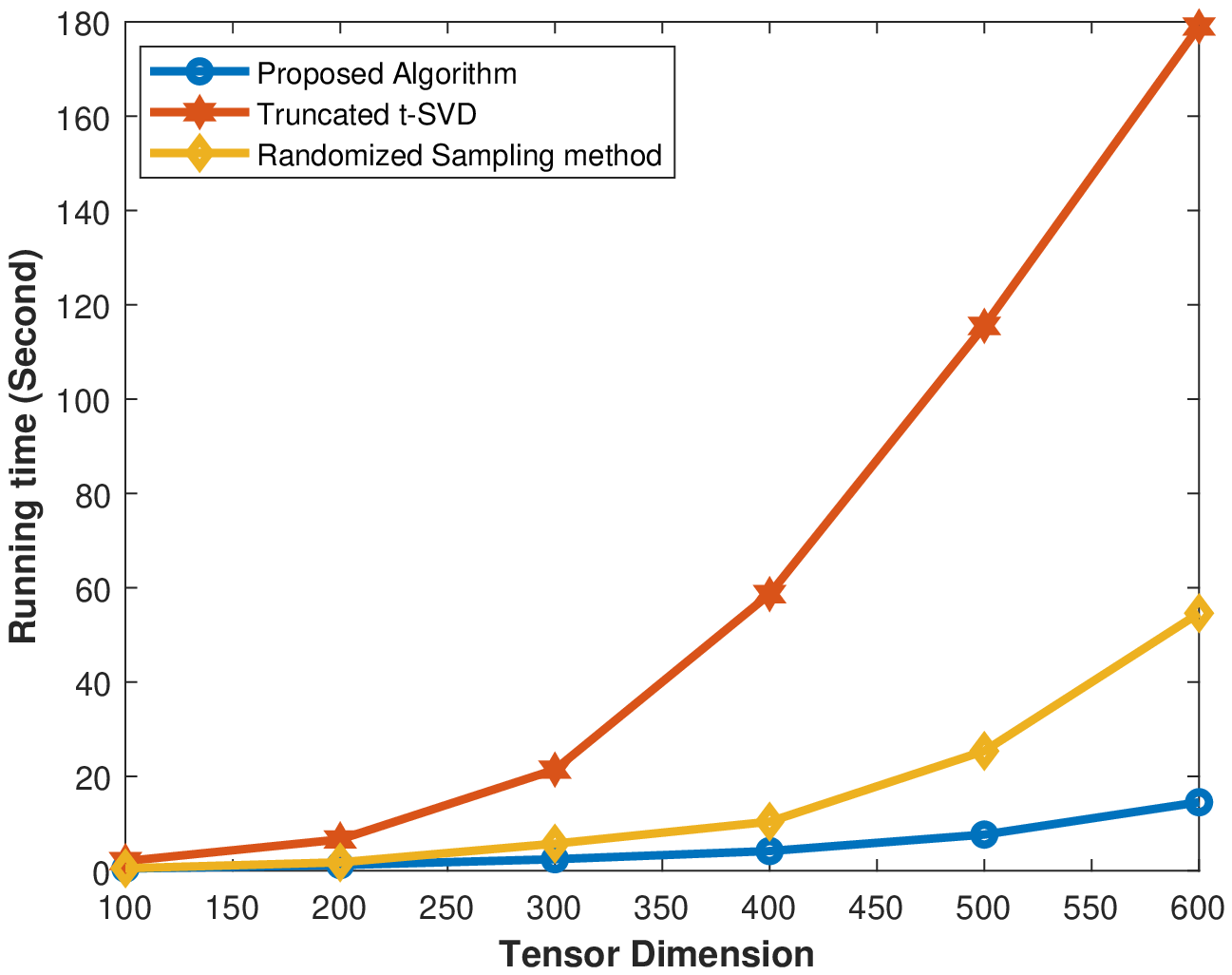}
    \includegraphics[width=0.49\linewidth]{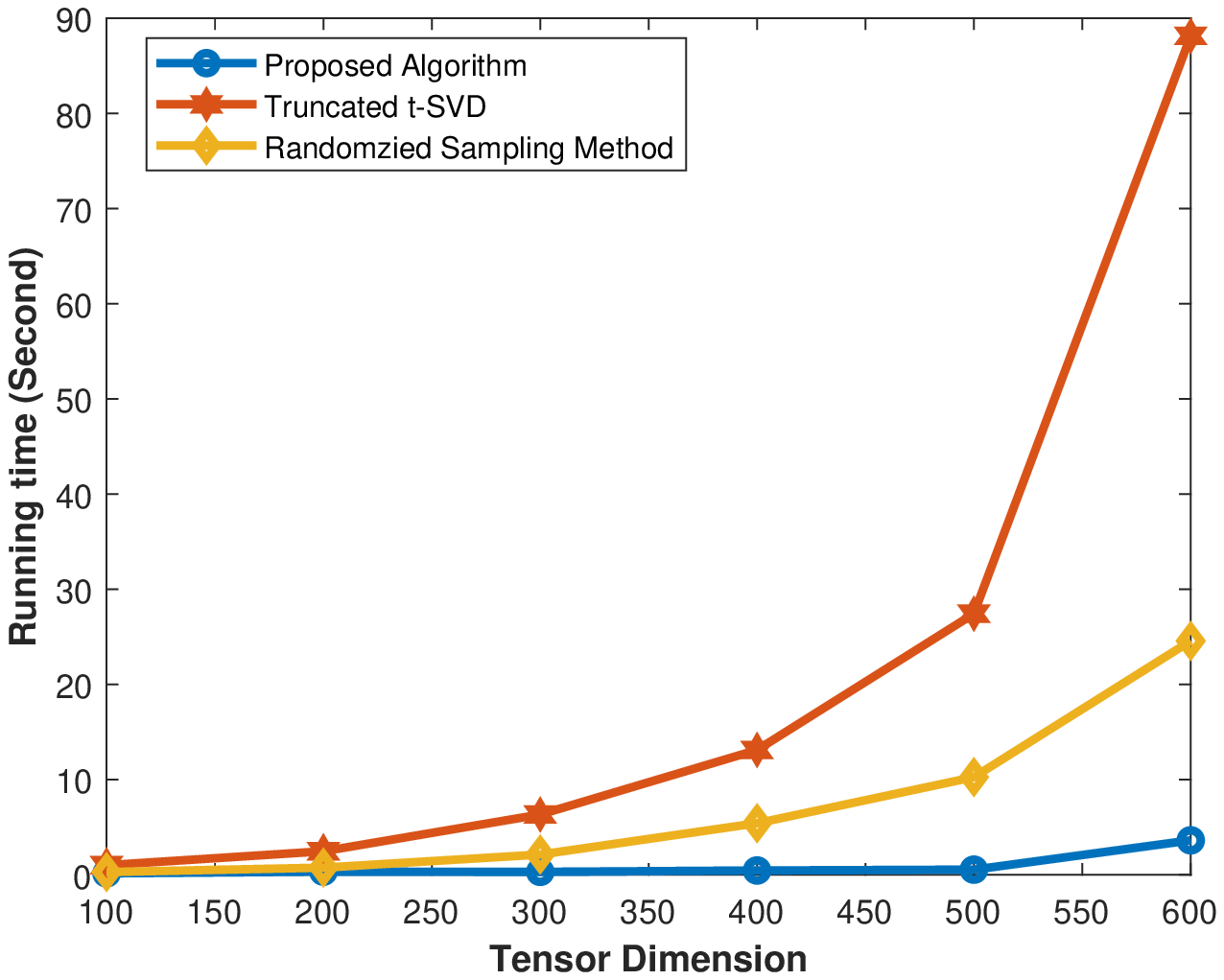}
    \includegraphics[width=0.49\linewidth]{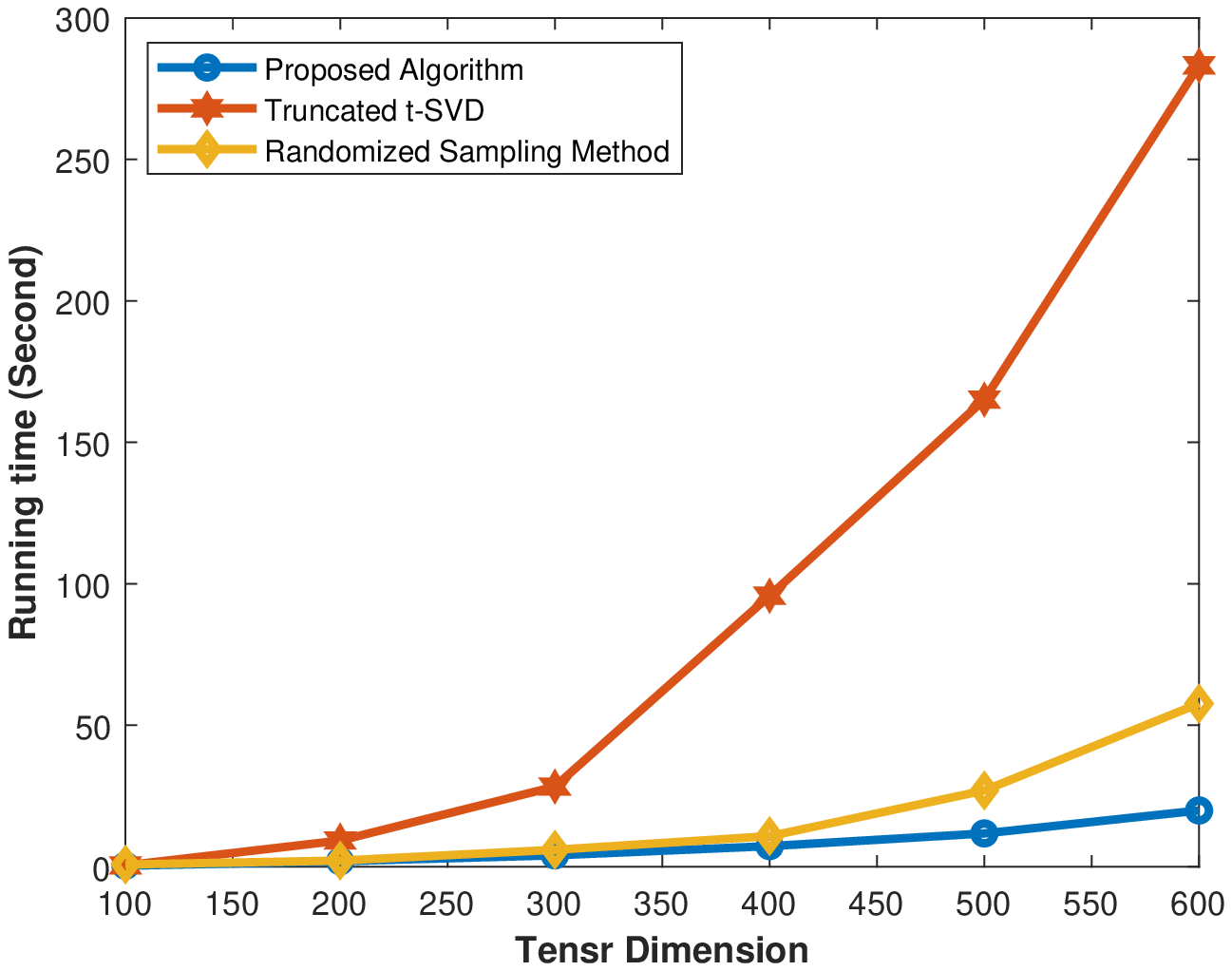}
    \caption{The running time comparison of the truncated t-SVD and the proposed algorithm for Case I (upper left), Case II (upper right) and Case III (bottom) for Example \ref{exm2}.}\label{fig2}
    \end{center}
\end{figure}

\begin{figure}
\begin{center}
    \includegraphics[width=0.7\linewidth]{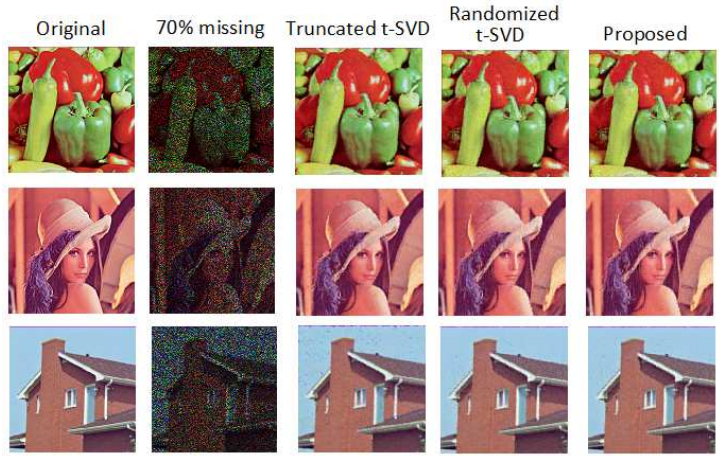}
    \caption{The original image, the available with 70\% missing pixels (randomly) and the reconstructed images using the truncated t-SVD and the proposed approach for Example \ref{exm3}.}\label{fig3}
    \end{center}
\end{figure}

\end{exa}

\begin{exa}\label{exm2}
In this example, we apply Algorithm \ref{CrossTensor_Alg} to compute low tubal rank approximations of function based tensors. To do so, we consider the following case studies:
\begin{itemize}
\item Case study I: $\quad\cX(i,j,k)=\frac{1}{(i^2+j^2+k^2)^{1/2}};$
    
\item Case study II: $\quad\cX(i,j,k)={\sin{(i+j+k)}}+\tanh(i+j+k);$
    
\item Case study III: $\quad\cX(i,j,k)=\frac{1}{(i^5+j^5+k^5)^{1/5}};$
\end{itemize}
where $1\leq i,j,k\leq n$ for $n=100,200,\ldots,600.$ It is not difficult to see that these tensors have low tubal ranks. The numerical tubal rank for case studies I, II and II for a tensor of size $100\times 100\times 100$ were 25, 5 and 43, respectively. However, for larger sizes the numerical tubal rank may be slightly changed. We applied the proposed approach to the mentioned data tensors in a similar way as for Example 1, to find the numerical tubal rank and the corresponding low tubal rank approximation. Here, for the case studies I, II and III, the proposed algorithm for $\epsilon=10^{(-8)}$ gave tubal ranks 24, 5 and 42, respectively, which is very close to the true numerical tubal ranks. Then for these tubal ranks, we applied the truncated t-SVD and the randomized t-SVD to compute a low tubal rank approximation. The running time and relative errors of the solutions of the algorithms are compared in Figure \ref{fig2} and Table \ref{Table1}, respectively. In view of Figure \ref{fig2}, the performance of the proposed algorithm compared with the truncated t-SVD and the randomized t-SVD is visible. These two experiments verified that the proposed approach is applicable for large-scale tensors because it works only on a small a part of the data tensor at each iteration while the classical approaches, e.g. the truncated t-SVD deals with the whole data tensor. The results in Table \ref{Table1} also show that the proposed algorithm provides an approximation with almost the same accuracy as the truncated t-SVD, which is known to be the best approximation in the lease-squares sense\footnote{For any unitary
invariant tensor norm.} for the low tubal rank approximation \cite{kilmer2011factorization}. 

% \begin{figure}
% \begin{center}
%     \includegraphics[width=1.2\linewidth]{Exm_2.png}
%     \caption{Example \ref{exm2}. The running time comparison of the proposed algorithm for Case I (upper left), Case II (upper right) and Case III (bottom).}\label{fig2}
%     \end{center}
% \end{figure}

\begin{table}
\begin{center}
\caption{Relative errors of results obtained by the truncated t-SVD and the proposed Algorithm for Example \ref{exm2}.}\label{Table1}
\smaller
\begin{tabular}{||c| c c c c c c||} 
\multicolumn{7}{c}{Case study I}\\
 \hline
 \diagbox[width=7em]{Methods}{N}& 100 & 200 & 300 & 400 & 500 & 600\\
 \hline\hline
 Truncated t-SVD \cite{kilmer2011factorization} & 1.6e-14 & 8.01e-13 & 6.9e-12 & 2.4e-11  & 5.9e-11 & 1.1e-10\\ 
 Randomized t-SVD \cite{tarzanagh2018fast} &  4.9e-14 & 8.4e-12 & 4.4e-11 & 1.5e-10 & 3.7e-10 & 7.7e-10\\ 
 Proposed algorithm  &3.1e-14 &  1.01e-12 & 3.5e-11& 2.95e-10 & 3.2e-10 & 9.7e-10\\
\hline
 \multicolumn{7}{c}{Case study II}\\ 
 \hline\hline
 Truncated t-SVD \cite{kilmer2011factorization} & 1.0e-15 & 1.6e-15 & 1.3e-15 & 1.7e-15 &  2.09e-15 & 2.5e-15 \\ 
  Randomized t-SVD \cite{tarzanagh2018fast} & 1.2e-15 & 1.5e-15 & 1.6e-15 & 2.4e-15 & 2.06e-15 &  2.3e-15  \\ 
 Proposed algorithm  &1.3e-15 & 1.5e-15 & 1.7e-15  & 1.8e-15 & 2.06e-15 & 2.3e-15  \\
 \hline
 \multicolumn{7}{c}{Case study III}\\
 \hline\hline
 Truncated t-SVD \cite{kilmer2011factorization} & 2.7e-14 & 1.3e-11 & 1.3e-10 &  5.04e-10 & 1.1e-09 &2.2e-09 \\ 
  Randomized t-SVD \cite{tarzanagh2018fast} & 1.9e-13 & 1.2e-10 & 1.01e-09 & 3.5e-09 & 6.5e-09 & 1.3e-08 \\ 
 Proposed algorithm  & 5.3e-14 & 7.7e-11 & 4.7e-10 & 7.1e-09 & 3.2e-09 & 2.1e-08\\
 \hline
\end{tabular}
\end{center}
\end{table}

\end{exa}

\begin{exa}\label{exm3}
{\bf Application to tensor completion.} In this example, we show the application of the proposed adaptive algorithm for the task of tensor completion. To this end, we consider the benchmark images ``Peppers'', ``Lena'' and ``House'' that of size $256\times 256 \times 3$ depicted in Figure \ref{fig3} (left) and remove 70\% of their pixels randomly shown in Figure \ref{fig3} (middle). We use the Peak signal-to-noise ratio (PSNR) to compare the performance of the proposed algorithm with the benchmark algorithm. The tensor decomposition formulation \eqref{MinRankCompl2} for the tensor completion problem is written as follows
% Besides if the uncompleted data tensor $\underline{\bf M}$ is corrupted by a level of noise then the following minimization problem is considered
% \begin{equation}\label{MinRankCompl2}
% \begin{array}{cc}
% \displaystyle \min_{\underline{\bf X}} & {\rm Rank} \left(\underline{\bf X}\right)\\
% \textrm{s.t.} & \|{{\bf P}_{\underline{\bf\Omega}} }({\underline{\bf X}})-{{\bf P}_{\underline{\bf\Omega}} }({\underline{\bf M}})\|_F\leq \epsilon\\
% \end{array}
% \end{equation}
\begin{equation}\label{MinRankCompl2}
\begin{array}{cc}
\displaystyle \min_{\underline{\bf X}} & {\|{{\bf P}_{\underline{\bf\Omega}} }({\underline{\bf X}})-{{\bf P}_{\underline{\bf\Omega}} }({\underline{\bf M}})\|^2_F},\\
\textrm{s.t.} & {\rm rank}(\underline{\bf X})=R,\\
\end{array}
\end{equation}
where the unknown tensor $\underline{\bf X}$ to be determined, and we assume that it has low tensor rank representation, $\underline{\bf M}$ is the original data tensor and ${\underline{\bf\Omega}}$ is the index of known pixels. The projector ${\bf P}_{\underline{\bf \Omega}}$ is defined as follows
\[
{\bf P}_{\underline{\bf \Omega}}(\X)=
\left\{
	\begin{array}{ll}
		\X({\bf i})  &  {\bf i}\in\underline{\bf\Omega},\\ 
		0 & {\bf i}\notin\underline{\bf\Omega},
	\end{array}
\right.
\]
where ${\bf i}=(i_1,i_2,\ldots,i_N)$ is an arbitrary multi-index with $1\leq i_n\leq I_n,\,\,n=1,2,\ldots,N$. Here, different kinds of tensor ranks and associated tensor decompositions can be considered in the formulation \eqref{MinRankCompl2}. The solution to the minimization problem \eqref{MinRankCompl2} can be approximated by the following iterative procedure 
\begin{equation}\label{Step1}
\underline{\mathbf Y}^{(n)}= \mathcal{L}(\underline{\mathbf X}^{(n)}),
\end{equation}
\begin{equation}\label{Step2}
\underline{\mathbf X}^{(n+1)}=\underline{\mathbf \Omega}\oast\underline{\mathbf Y}^{(n)}+(\underline{\mathbf 1}-\underline{\mathbf \Omega})\oast\underline{\mathbf Y}^{(n)},
\end{equation}
as described in \cite{ahmadi2022cross} to complete the unknown pixels where $\mathcal{L}$ is an operator, which computes a low-rank tensor approximation of the data tensor $\underline{\mathbf X}^{(n)}$,  $\underline{\mathbf 1}$ is a tensor whose all components are equal to one and $\oast$ is the Hadamard (elementwise) product. For the low-rank computations in the first step \eqref{Step1}, we apply the proposed Algorithm \ref{CrossTensor_Alg} with a given number of iterations and not a given tolerance to find the lateral and horizontal slice indices and compute the approximation $\underline{\bf C}*\underline{\bf U}*\underline{\bf R}$ where $\underline{\bf U}=\underline{\bf C}^{\dag}*\underline{\bf X}*\underline{\bf R}^{\dag}$ and $\underline{\bf C},\,\underline{\bf R}$ are the sampled lateral and horizontal slices, respectively. The tubal rank $R=70$, was used in our computations. Beside applying the proposed algorithm, we also used the truncated t-SVD and the randomized t-SVD in our computations. The reconstructed images using the proposed approach and the truncated t-SVD are displayed in Figure \ref{fig3} (bottom). The running time required to compute these reconstructions and also their PSNR are reported in Table \ref{Table1}. The results in Table \ref{Table1} and Figure \ref{fig3}, clearly illustrate that the proposed adaptive algorithm provides comparable results in much less running time. This clearly shows the feasibility and efficiency of the proposed algorithm for fast tensor completion task. 

\begin{table}
\begin{center}
\caption{Running time (second ) and PSNR of the constructed images using the truncated t-SVD and the proposed Algorithm for Example \ref{exm3}, Time (s) and PSNR (dB).}\label{Table1}
\vspace{0.2cm}
\begin{tabular}{||c| c c | c c | c c||} 
 \hline
 \cline{2-7}
 & \multicolumn{2}{c}{{Peppers}} & \multicolumn{2}{c}{{Lena}} & \multicolumn{2}{c||}{{House}}\\
  \cline{2-7}
  & Time & PSNR & Time & PSNR & Time & PSNR\\
 \hline\hline
 Truncated t-SVD \cite{kilmer2011factorization} & 55.03  & { 27.74} & 45.2  & { 27.36} & 61.94  & 28.05 \\ 
  Randomized t-SVD \cite{tarzanagh2018fast}& 10.12  & 27.55 & 11.23  & 27.40 & 10.83  & 29.40 \\ 
 Proposed algorithm  & { 7.60 } & 27.55 & { 8.34 } & { 27.46} & { 7.27 } & { 29.59}\\
 \hline
\end{tabular}
\end{center}
\end{table}

\end{exa}

\begin{exa}\label{Ex_4}
{
{\bf Application to PEdesTrian Attribute Recognition task.}
In this experiment, we show an application of the proposed method for the Pedestrian Attributes Recognition (PAR) task \cite{wang2022pedestrian}. We consider the PEdesTrian Attribute dataset (PETA) dataset \cite{deng2014pedestrian}, which was widely used in the literature for the PAR problem. It includes 8705 different persons in the 19000 pedestrian images, which have 65 attributes (61 binary and 4 multi-class). Although the images are not all the same size, we resize them in this experiment to $256\times 128\times 3$. We only take into account $1012$ images and construct a fourth-order tensor with the size $256\times 128\times 3\times 1012$ and reshape it into a third-order tensor with the size $256\times 384\times 1012$. With an error bound of $\epsilon= 0.1$, we applied the suggested approach to the above dataset to determine the relevant tubal-rank. The truncated t-SVD of the underlying dataset was then computed for this tubal-rank. The reconstructed images that were obtained by them for two random samples using the proposed algorithm and compared to the truncated t-SVD are shown in Figure \ref{fig_recon}. Here we achieved $\times 5$ speed-up compared to the truncated t-SVD. The results unequivocally show that the suggested technique can produce similar results in less computational time. Additionally, the Attribute-specific Localization (ASL) model \cite{tang2019improving}, an effective deep neural network (DNN), was taken into account as it had provided cutting-edge results for the PAR problem. In order to create a lightweight model with fewer parameters and complexity \cite{Ashish2023}, we first compressed the underlying convolution layers in the ASL model using the Error Preserving Correction-CPD \cite{phan2018error}  and the SVD. As our test datasets, we also compressed $30\%$ of the PETA dataset's images using Algorithm \ref{CrossTensor_Alg} (for $\epsilon= 0.1,0.2,$ and $0.3$) and compared the ASL model's performance in identifying pedestrian features in the compressed and original images.
}{Table \ref{Table_new} displays the experiment's outcomes. We see that the light-weight model's accuracy for both the original and compressed images is quite similar. It should be noted that the model was not trained on compressed photos, though one may do so to improve recognition accuracy. As a result, this concept may be applied to Internet of Things (IoT) applications where tremendous amounts of data in various shapes and formats are generated (for example, image sensors embedded in mobile cameras produce enormous amounts of data in the form of higher-resolution photographs and videos). Here, it is fundamental to install compacted DNNs and portable DL models on the edge of the IoT network, along with having fast data communication for real-time applications (denoising, defogging, deblurring, segmentation, target detection, and recognition). Using the suggested method, we may use the compressed form of the data in these applications.}

\begin{figure}
\begin{center}
    \includegraphics[width=0.50\linewidth]{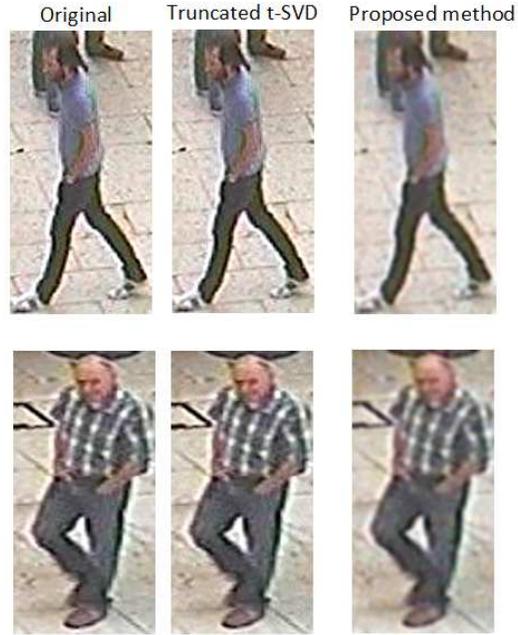}
    \caption{Comparing the original images and their compressed forms by the proposed algorithm and the truncated t-SVD for Example \ref{Ex_4}.}\label{fig_recon}
    \end{center}
\end{figure}

\begin{table}
\begin{center}
\caption{Comparing the running times (second) and relative errors achieved by the proposed algorithm and the Truncated t-SVD \cite{kilmer2011factorization} with original accuracy {\bf 0.8887} for Example \ref{Ex_4}.}\label{Table_new}
\begin{tabular}{||c| c |c||} 
 \multicolumn{3}{c}{$\epsilon=0.3$}
 \\\hline
Algorithms  & Running Time (s) & Recognition accuracy\\
 \hline\hline
 Truncated t-SVD \cite{kilmer2011factorization} &  236.45 &  { 85.86\%}\\ 
 Proposed algorithm  & { 45.56} & 84.36\%  \\
 \hline
  \multicolumn{3}{c}{$\epsilon=0.2$}\\
  \hline
  Truncated t-SVD \cite{kilmer2011factorization} & 196.45  & { 87.21\%} \\ 
 Proposed algorithm  & { 35.97} &  86.16\% \\
 \hline
  \multicolumn{3}{c}{$\epsilon=0.1$}\\
  \hline
  Truncated t-SVD \cite{kilmer2011factorization} &  150.32 & { 88.42\%}\\ 
 Proposed algorithm  & { 27.12}   & 87.21\%  \\\hline
\end{tabular}
\end{center}
\end{table}

\end{exa}

\section{Conclusion and future works}\label{Sec:Con}
In this work, we proposed an adaptive tubal tensor approximation algorithm for the computation of the tensor SVD. The proposed algorithm can estimate the tubal rank of a tensor and provide the corresponding low tubal rank approximation. The experimental results verified the feasibility of the proposed algorithm. Our future work will be developing a blocked version of the proposed adaptive tubal tensor algorithm. The block version can be further improved using the parallel hierarchical strategy \cite{liu2020parallel} and we will investigate this in future works. In the matrix case, it is known that the  maximum volume (maxvol) algorithm as a matrix cross approximation method provides close to optimal low-rank approximations. Generalization of the maxvol approach from the matrix case to tensors based on the t-product is our ongoing research work.

\section{Acknowledgement} The authors would like to thank the editor and two reviewer reviewers for their constructive comments, which have greatly improved the quality of the paper. The
work was partially supported by the Ministry of Education and Science of the Russian
Federation (grant 075.10.2021.068).

\section{Conflict of Interest Statement}
The authors declare that they have no
conflict of interest with anything.

\bibliographystyle{elsarticle-num} 
\bibliography{cas-refs}

% \bibitem{}

% \end{thebibliography}
\end{document}